\documentclass{amsart}

\usepackage{lmodern}
\usepackage{amsmath,amsfonts,amssymb, amsthm, xcolor, enumitem, graphicx, hyperref}
\usepackage{url}
\usepackage{appendix}

\newtheorem{theorem}[subsection]{Theorem}

\newtheorem{corollary}[subsection]{Corollary}

\theoremstyle{definition} 

\newtheorem{remark}[subsection]{Remark}

\newcommand{\bra}{\langle}
\newcommand{\ket}{\rangle}
\newcommand{\tens}[1]{\mathbin{\mathop{\otimes}\limits_{#1}}}

\begin{document}


\title[Gateaux derivative in operator spaces]{Gateaux derivative of matrix norms in operator spaces and operator systems}
\author{Sushil Singla}
\address{Department of Mathematics and Statistics, University of Regina, Regina, Saskatchewan S4S 0A2, Canada; Corvinus Institute for Advanced Studies (CIAS), Corvinus University of Budapest, Hungary.}
\curraddr{}
\email{sushil@uregina.ca; sushil.sushil@uni-corvinus.hu}
\thanks{Supported in part by the PIMS Postdoctoral Fellowship Program, by the Faculty of Science, University of Regina, and by the International Research Excellence Fellowship Programme of Corvinus Institute for Advanced Studies (CIAS), Corvinus University of Budapest, Hungary.}


\begin{abstract}
	We find expressions for the Gateaux derivative of the matrix norms in operator spaces, and operator systems. Some applications of the results to quantum probability measures, states on C$^*$-algebras, and Birkhoff-James orthogonality are also presented.
\end{abstract}

\maketitle


\section{Introduction}

If $(\mathcal X, \|\cdot\|)$ is a normed vector space, then the norm function $x\mapsto \|x\|$ is
sublinear, and in particular, convex. Thus, for all $x,y\in \mathcal X$,
the one-sided limit
\[
\lim_{t\rightarrow 0^+}\dfrac{\|x+ty\|-\|x\|}{t}
\]
exists, 
and the resulting quantity is
known as the \emph{Gateaux derivative} of the norm function $\|\cdot\|$ at $x$ in the direction of $y$. 
One source of interest in the Gateaux derivative is its role in Banach space geometry, where the \emph{Birkhoff-James orthogonality} of vector $x$ to the vector $y$
admits descriptions involving the Gateaux derivative of the norm at $x$ in the direction of $y$, see \cite[Theorem 1.4]{keckic2004}. The notion of Gateaux derivative has been used in literature for commutator approximants (see \cite{kittaneh1996, maher1992}), and in results related with the smooth points and extreme points of the norm's unit ball (see \cite{abatzoglou1979}).

Previous works have considered the Gateaux derivative of the norm function for C$^*$-norms \cite{singla2021} and the Schatten $p$-norms \cite{keckic2004}.
In the present paper, we consider the matrix norms $\|\cdot\|_n$ that arise from an operator space or an operator system.

A key idea is the notion of support mapping on an operator space $(V, \|\cdot\|_n)$. Let $M_n$ denote the space of $n\times n$ complex matrices. For $A\in M_n$, let $\Re(A):=(A+A^*)/2$
and 
$\rho^* : M_n\rightarrow\mathbb R$ be the function 
defined by 
\[
\rho^*(A)=\inf\{t\geq 0 \,\,|\,\, \Re(A)\leq tI_n\},  
\]
where $I_n\in M_n$ is the identity matrix.  If $v\in M_n(V)$, the normed vector space of matrices with entries from $V$,
then a linear map 
$\phi : V\rightarrow M_n$ is said to be a \emph{support mapping for $v$} if, for all $k\in\mathbb N$ and $w\in M_k(V)$, 
we have 
\[
\rho^*(\phi_k(w))\leq \|w\|_k
\]
and 
\[
\rho^*(\phi_n(v))=\|v\|_n,
\]
where, for any $\ell\in\mathbb N$,
 $\phi_\ell$ denotes the linear map $M_\ell(V)\rightarrow M_\ell(M_n)$ given by $[v_{ij}]_{i,j=1}^{\ell}\mapsto [\phi(v_{ij})]_{i,j=1}^{\ell}$ for all $v_{ij}\in V$.
The existence of support mappings on $(V, \|\cdot\|_n)$, for every $v\in M_n(V)$ and every $n\in\mathbb N$, 
was established by Winkler in \cite[Theorem 3.3.2]{winkler1996}. 

For $A\in M_n$, we have $\Re(A)\leq tI_n$ if and only if $\langle\Re(A)\eta, \eta\rangle\leq t$ for all unit vectors $\eta\in\mathbb C^n$. Thus, we have \[
\rho^*(A)=\sup\{\langle\Re(A)\eta, \eta\rangle \,\,|\,\, \eta\in\mathbb C^n, \|\eta\|=1\}.\]
For $A\in M_n$, we define 
\[
	M_0(A)=\{\eta\in \mathbb C^n \,\, |\,\, \|\eta\|=1\text{ and }\bra \Re(A)\eta,\eta\ket=\|A\|\}.
\] 
We note that the above set equals the unit ball of the eigenspace of $A$ corresponding to eigenvalue $\|A\|$, i.e.
\[
M_0(A)=\{\eta\in \mathbb C^n \,\, |\,\, \|\eta\|=1\text{ and }A\eta=\|A\|\eta\}.
\]
To see this, we note that for a unit vector $\eta$, the condition $A\eta=\|A\|\eta$ is equivalent to $\bra \Re (A)\eta,\eta\ket=\|A\|$. We note that $A\eta=\|A\|\eta$ implies $\bra A^*A\eta, \eta\ket =\|A\|^2$, consequently $\bra A^*\eta, \eta\ket =\|A\|$ that further implies $\bra \Re(A)\eta,\eta\ket=\|A\|$. To prove the converse, we note if $\bra \Re(A)\eta,\eta\ket=\|A\|$, then  Cauchy-Schwarz inequality implies $$\|A\|=\bra \Re(A)\eta,\eta\ket \leq |\bra A\eta,\eta\ket|\leq  \|A\eta\|\leq \|A\|.$$ So, equality occurs throughout and using the condition of equality in Cauchy-
Schwarz inequality, we get $A\eta=\|A\|\eta$.

The first goal in this paper is to express the Gateaux derivative of matrix norms $\|\cdot\|_n$ that arises from an operator space in the manner presented in Theorem \ref{thm1} below.
\begin{theorem}\label{thm1}
	If $(V, \|\cdot\|_n)$ is an operator space, and 
	$v,w\in M_n(V)$. Then there exists a support mapping for $v$ such that $M_0(\phi_n(v))\neq \emptyset$ and
	 \begin{align*}\lim_{t\rightarrow 0^+}\dfrac{\|v+tw\|_n-\|v\|_n}{t}=\max\{\bra\Re(\phi_n(w))\eta, \eta\ket \,\, |\,\,  &\phi : V\rightarrow M_n \text{ is a support mapping }\\
		&\text{for } v\text{ and }\ \eta\in M_0(\phi_n(v))\}.
		\end{align*} 
\end{theorem}

 If $\mathcal S$ be an operator system with Archimedean order unit $e$, then a \emph{matrix state} on $\mathcal S$ 
is a unital completely positive map $\phi :\mathcal S\rightarrow M_k$ such that $\phi(e)=I_k$. 
For every $k\in\mathbb N$, let $S_k(\mathcal S)$ denote the set of matrix states $\phi :\mathcal S\rightarrow M_k$. For $s\in\mathcal S$, the matrix range of $s$ is the union of the sets 
\[
W_k(s)=\{A\in M_k \,\,|\,\, A=\phi(s)\ \text{ for some }\ \phi\in S_k(\mathcal S)\}.
\]
If $\mathcal S$ is the space of bounded linear operators on a Hilbert space with identity operator $I$, then for any bounded linear operator $T$ and $n=1$, $W_1(T)$ is the topological closure of the classical numerical range of $T$. It was proved by Lumer \cite{lumer1961} that
\[
\sup\{\mathrm{Re}(\lambda)\,\,|\,\, \lambda\in W_1(T)\}=\lim\limits_{t\rightarrow 0^+}\dfrac{\|I+tT\|-1}{t},
\] where $\mathrm{Re}(a+ib)=a$ for all $a,b\in\mathbb R$.
This was extended to matrix ranges by Narcowich and Ward \cite{narcowich--ward1982} as follows. For a unital
C$^*$-algebra $\mathcal A$, let $\|\cdot\|_{\sigma}$ denote the function $x\mapsto \|x\|_\sigma$ on $M_n(\mathcal A)$ defined by 
\[
\|x\|_\sigma=\sup\left\{\sqrt{s_\phi(x^*x)}\,\,|\,\, \phi\in\mathcal S_n(\mathcal A) \right\},\quad \text{for}\quad x\in M_n(\mathcal A),
\]
where $s_\phi : M_n(\mathcal A)\rightarrow\mathbb C$ is the functional 
\[
s_\phi([x_{ij}]_{i,j=1}^n)=\dfrac{1}{n}\sum_{1\leq i,j\leq n}\ \phi(x_{ij})_{ij},
\] where $\phi(x_{ij})_{ij}$ denotes the $(i,j)$-entry of the matrix $\phi(x_{ij})\in M_n$. 
By \cite[Proposition 3.1, Corollary 3.4]{narcowich--ward1982}
the function $x\mapsto \|x\|_\sigma$ defines a norm on $M_n(\mathcal A)$, and it has the property
that, for all $x\in \mathcal A$ and $B=[b_{ij}]_{i,j=1}^n\in M_n$, 
\[
\max\{\mathrm{Re}(\mathrm{trace}(A^tB))\,\, |\,\, A\in W_n(x)\}=n\cdot\lim\limits_{t\rightarrow 0^+}\dfrac{\|e_n+t(\overline{b}_{ij}x)_{ij}\|_\sigma-1}{t}.
\]

The following result gives a slightly different extension of Lumer's formula.

\begin{theorem}\label{cor3}
	Let $\mathcal S$ be an operator system, and let $s_1, s_2\in M_n(\mathcal S)$. Then, we have 
\begin{align}\label{e:qaz}
\lim\limits_{t\rightarrow 0^+}\dfrac{\|s_1+ts_2\|_n-\|s_1\|_n}{t} = \frac{1}{\|s_1\|_n}\max\{& \mathrm{Re}(\bra\phi_n(s_2)\eta, \phi_n(s_1)\eta\ket) \,\, |\,\,  \phi \in S_n(\mathcal S), \eta\in\mathbb C^{n^2}\\
&\|\eta\|=1,\|\phi_n(s_1)\|=\|s_1\|_n= \|\phi_n(s_1)\eta\|\}\nonumber.
\end{align}
\end{theorem}

Let $e_n\in M_n(\mathcal S)$ denotes the diagonal matrix with $e$ at each diagonal entry. Then for $s\in M_n(\mathcal S)$, Theorem \ref{cor3} reads as \[\lim\limits_{t\rightarrow 0^+}\dfrac{\|e_n+ts\|_n-1}{t} = \max\{\mathrm{Re}(\lambda)\,\, |\,\, \lambda\in W_1(\phi_n(s)),\ \phi \in S_n(\mathcal S)\}.\]
Note that $\phi\in S_n(\mathcal S)$ has range in $M_n$, the elements $\lambda$ in the above equation is taken from the numerical range of the $n^2\times n^2$ matrix
$\phi_n(s)$.

If we define $M_1(A)=\{\eta\in \mathbb C^n \,\, |\,\, \|\eta\|=1\text{ and }\|A\eta\|=\|A\|\}$, then \eqref{e:qaz} can be rewritten as: \begin{align*}
		\lim\limits_{t\rightarrow 0^+}\dfrac{\|s_1+ts_2\|_n-\|s_1\|_n}{t} = \frac{1}{\|s_1\|_n}\max\{& \mathrm{Re}(\bra\phi_n(s_2)\eta, \phi_n(s_1)\eta\ket) \,\, |\,\,  \phi \in S_n(\mathcal S),\\
		&\eta\in M_1(\phi_n(s_1)), \|\phi_n(s_1)\|=\|s_1\|_n\}.
	\end{align*} It is important to note that the set $M_1(A)$ and $M_0(A)$ are not equal in general since $\|A\eta\|=\|A\|$ may not imply $A\eta=\|A\|\eta$, for example for $A=-I$. But equality of the sets $M_1(A)$ and $M_0(A)$ may occur in some cases, for example, it occurs when $A$ is positive semidefinite matrix.

In Section \ref{proofs}, we will provide the proofs of Theorem \ref{thm1} and Theorem \ref{cor3}. In Section \ref{section3}, applications of these results will be given in the commutative C$^*$-algebras and simple finite-dimensional C$^*$-algebras. A few new results, for instance Theorem \ref{thm3} and Theorem \ref{thm5}, that are useful in the geometry of Banach spaces, are also mentioned. 

\section{Proofs of the Main Results}\label{proofs}

\subsection{Proof of Theorem \ref{thm1}}

Our assumptions are that $V$ is an operator space with matrix norms $\|\cdot\|_n$, and that $n\in\mathbb N$, $v,w\in M_n(V)$ are fixed.
We aim to 
show the Gateaux derivative of $\|\cdot\|_n$ at $v$ in the direction of $w$, namely the quantity
\begin{equation}\label{e:gd1}
\lim_{t\rightarrow 0^+} \dfrac{\|v+tw\|_n-\|v\|_n}{t},
\end{equation}
 is given by
\begin{equation}\label{e:gd2}
\sup\left\{\bra\Re(\phi_n(w))\eta, \eta\ket \, |\, \phi  \text{ is a support mapping for } v \text{ and }\eta\in M_0(\phi_n(v))\right\},
\end{equation}
and the supremum of the set in \eqref{e:gd2}  is also attained. 

We first notice that the quantity in \eqref{e:gd2} is less than or equal to the quantity in \eqref{e:gd1}. To prove this inequality, we use the convention that $\sup\emptyset=-\infty$ allowing the possibility that $M_0(\phi_n(v))=\emptyset$ for the moment.

Because
$\|\cdot\|_n$ is a convex function, we have 
\[
\lim_{t\rightarrow 0^+}\dfrac{\|v+tw\|_n-\|v\|_n}{t}=\inf_{t\geq 0}\left\{\dfrac{\|v+tw\|_n-\|v\|_n}{t}\right\}.
\]
If $\phi : V\rightarrow M_n$ is a support mapping for $v$ and $\eta\in M_0(\phi_n(v))$, then
	\begin{align*}
		\|v+tw\|_n&\geq \rho^*(\Re(\phi_n(v+tw)))\\
			&=\sup\{\langle\Re(\phi_n(v+tw))\eta, \eta\rangle \,\,|\,\, \eta\in\mathbb C^n, \|\eta\|=1\}\\
			&\geq \sup\{\bra\Re(\phi_n(v+tw))\eta, \eta\ket \,:\,\eta\in M_0(\phi_n(v))\}\\
			&=\sup\{\bra\Re(\phi_n(tw))\eta, \eta\ket \,:\,\eta\in M_0(\phi_n(v))\}+\|v\|_n,
	\end{align*}
which establishes the inequality
\[
\dfrac{\|v+tw\|_n-\|v\|_n}{t}\geq \sup\{\bra\Re(\phi_n(w))\eta, \eta\ket \,:\,\eta\in M_0(\phi_n(v))\}.
\]
By considering the infimum over all $t\geq 0$ and the supremum over all support functionals $\phi : V\rightarrow M_n$ of $v$, 
we deduce the quantity in (\ref{e:gd2}) is less than or equal to the quantity in (\ref{e:gd1}).

To prove the complementary inequality, 
let 
\[
a=\lim_{t\rightarrow 0^+}\dfrac{\|v+tw\|_n-\|v\|_n}{t}.
\]

Since the norm function is convex, the corresponding slope function is non-decreasing, consequently, we have 
\begin{equation}\label{new2}
\dfrac{\|v-tw\|_n-\|v\|_n}{-t}\leq a\leq \dfrac{\|v+tw\|_n-\|v\|_n}{t}.
\end{equation}
Now, if we consider the real linear functional $G$ on the two dimensional subspace generated by $v$ and $w$, that maps $v$ to $\|v\|_n$ and $w$ to $a$, then using \eqref{new2}, we get $|G(v+tw)|\leq \|v+tw\|_n$. By the Hahn-Banach Theorem, we can extend $G$ to a real linear functional on $V$ such that $\|G\|=1$, $G(v)=\|v\|_n$ and $G(w)=a$. Let $F\in (M_n(V))^*$ be the complexification of the real linear map $G$. Then, $F$ satisfies
\begin{equation}\label{step1}
\|F\|=1,\, \mathrm{Re}(F(v))=\|v\|_n,\text{ and } \mathrm{Re}(F(w))=a,
\end{equation} 
where $\mathrm{Re}(z)=(z+\overline{z})/2$ for all $z\in\mathbb C$. 
Arguing along the lines of the proof of \cite[Theorem 3.3.2]{winkler1996} (explained below), 
there is a support mapping $\phi : V\rightarrow M_n$ for $v$ and a 
unit vector $\eta\in (\mathbb C^n)^n$ such that 
\begin{equation}\label{laststep}
F(u)=\bra \phi_n(u)\eta, \eta\ket, \text{ for all } u\in V.
\end{equation} 
Thus, 
\[
a=\mathrm{Re}(\bra\phi_n(w)\eta, \eta\ket)=\bra \Re(\phi_n(w))\eta, \eta\ket \leq \max\{\bra\Re(\phi_n(w))\gamma, \gamma\ket \,:\,\gamma\in M_0(\phi_n(v))\}.
\]
Furthermore, the condition $$\|v_n\|=\mathrm{Re}(F(v))=\mathrm{Re}(\bra \phi_n(v)\eta, \eta\ket)=\bra \Re(\phi_n(v))\eta, \eta\ket$$ implies $\eta\in M_0(\phi_n(v))$.
 This proves that 
the quantity in \eqref{e:gd1} is less than or equal to the quantity in \eqref{e:gd2}, and furthermore the supremum in \eqref{e:gd2} is attained at the support mapping $\phi$ and the unit vector $\eta$ constructed above. This establishes the assertion stated in Theorem \ref{thm1}.

Therefore, to complete the proof, we adapt the argument in \cite[Theorem 3.3.2]{winkler1996} to show
there are a support mapping $\phi : V\rightarrow M_n$ for $v$ and 
unit vector $\eta\in (\mathbb C^n)^n$ such that $F(u)=\bra \phi_n(u)\eta, \eta\ket$, for all $u\in V$.

Assume
$F\in (M_n(V))^*$ is a linear functional with the properties exhibited in \eqref{step1}.
Using \cite[Lemma 2.3.5]{winkler1996} (or drawing from the proof of \cite[Lemma 2.3.2]{Effros--Ruan-book}), there is a state
$p$ on $M_n$ such that 
\begin{equation}\label{step2}\mathrm{Re }F(A^*uA)\leq p(A^*A)\|u\|_n,
\end{equation} 
for all $u\in M_n(V)$ and $A\in M_{k,n}$, for arbitrary $r\in\mathbb N$. 
Since $p$ is a linear
functional on $M_n$, there exist $\eta_1,\dots,\eta_n\in\mathbb C^n$ such that 
\[
p(A)=\sum\limits_{i=1}^n \bra A\eta_i,\eta_i\ket,
\]
for all $A\in M_n$. Letting $\eta=(\eta_1,\dots,\eta_n)$ with $\eta_i\in\mathbb C^n$ for $1\leq i\leq n$, we get $$p(A)=\bra (A\tens{}I_n)\eta, \eta\ket.$$ Since $p$ is a state on $M_n$, hence $1=p(I)=\bra\eta,\eta\ket$, this implies $\|\eta\|= 1$. Thus, $$p(A^*A)=\bra (A^*A\tens{}I_n)\eta,\eta\ket=\|(A\tens{}I_n)\eta\|^2.$$ Using \eqref{step2}, this implies \begin{equation}\label{step2a}|F(A^*uA)|\leq \left(\sqrt{2}\|u\|_n\right)\|(A\tens{}I_n)\eta\|^2,\end{equation} for all $A\in M_n$.
Let $B_u: \mathcal H\rightarrow\mathcal H$ be the sesquilinear form defined on the subspace $\mathcal H=(M_n\tens{}I_n)\eta$ of $\mathbb C^{n^2}$ as $$B_u((B\tens{}I_n)\eta, (A\tens{}I_n)\eta)=F(A^* uB).$$ Then, $B_u$ is a bounded sesquilinear form  using \eqref{step2a}, and thus as an application of Riesz Representation Theorem there exists a bounded operator  $\psi(u)$ on $\mathcal H$ such that $$F(A^*uB)=\bra \psi(u)(B\tens{}I_n)\eta, (A\tens{}I_n)\eta\ket.$$

	Clearly, $\psi$ is a linear map. If we consider $P$ to be the projection from $\mathbb C^{n^2}$ onto $\mathcal H$, then $P\in(M_n\tens{}I_n)'$ (since $\mathcal H$ is invariant under $M_n\tens{}I_n$). Let $\tilde{\phi} : V\rightarrow M_n$ be the map defined by $$\tilde{\phi}(u)=P\psi(u)P,$$ and this gives $$F(A^*uB)=\bra \tilde{\phi}(u)(B\tens{}I_n)\eta, (A\tens{}I_n)\eta\ket,$$ for all $A, B\in M_n$. In particular, we have  $$F(u)=\bra \tilde{\phi}(u)\eta, \eta\ket,\  \tilde{\phi}(uA)=\tilde{\phi}(u)(A\tens{}I_n), \text{ and }\tilde{\phi}(A u)=(A\tens{}I_n)\tilde{\phi}(u),$$ for all $u\in M_n(V)$ and $A\in M_n$. Using linear algebra techniques, these in turn implies that there exists $\phi : V\rightarrow M_n$ such that $\tilde{\phi}(u)=\phi_n(u)$. Thus, we get $$F(A^*uB)=\bra \phi_n(u)\eta, \eta\ket,$$ for all $A,B\in M_n$, proving \eqref{laststep} and, hence, the theorem.
\qed

\subsection{Proof of Theorem \ref{cor3}} We aim to prove, for $s_1, s_2\in M_n(\mathcal S)$, where
$\mathcal S$ is an operator system, that
\begin{align*}
	\lim\limits_{t\rightarrow 0^+}\dfrac{\|s_1+ts_2\|_n-\|s_1\|_n}{t} =  \frac{1}{\|s_1\|_n}\max\{& \mathrm{Re}(\bra\phi_n(s_2)\eta, \phi_n(s_1)\eta\ket) \,\, |\,\,  \phi \in S_n(\mathcal S), \eta\in\mathbb C^{n^2}\\
	&\|\eta\|=1,\|\phi_n(s_1)\|=\|s_1\|_n= \|\phi_n(s_1)\eta\|\}.
\end{align*}

Using Theorem \ref{thm1}, there exist a support mapping
$\psi :\mathcal S\rightarrow M_n$
for $s_1$ and a vector $\eta\in M_0(\psi_n(s_1))$ such that
\[
\lim\limits_{t\rightarrow 0^+}\dfrac{\|s_1+ts_2\|_n-\|s_1\|_n}{t}=\bra \Re(\psi_n(s_2))\eta, \eta\ket\quad \text{and}\quad \bra \Re(\psi_n(s_1))\eta,\eta\ket=\|s_1\|_n.
\]
Let $\mathcal A$ be a $C^*$-algebra containing $\mathcal S$ is an operator subsystem. 
By Wittstock's extension theorem \cite[Theorem 8.2]{Paulsen-book}, we can extend $\psi$ to a completely contractive map to $\mathcal A$ that we will denote by same symbol. Using Wittstock's decomposition theorem \cite[Theorem 8.5]{Paulsen-book}, there exists a Hilbert space $\mathcal K$, a $\ast$-representation $\pi : \mathcal A\rightarrow\mathcal B(\mathcal K)$, and linear isometries $V_1, V_2:\mathbb C^n\rightarrow\mathcal K$ such that 
\[
\psi(a)=V_1^\ast\pi(a)V_2\text{ for all }a\in\mathcal A.
\]
Thus, 
\begin{equation}\label{eq1}\lim\limits_{t\rightarrow 0^+}\dfrac{\|s_1+ts_2\|_n-\|s_1\|_n}{t}=\mathrm{Re}(\bra \pi_n(s_2)\widetilde{V}_2\eta, \widetilde{V_1}\eta\ket)
\end{equation} 
and  $\mathrm{Re}(\bra \pi_n(s_1)\widetilde{V}_2\eta,\widetilde{V_1}\eta\ket)=\|s_1\|_n$ where $\widetilde{V}_i=V_i\oplus\dots\oplus V_i$ for $i=1,2$. We note that $\mathrm{Re}(\bra \pi_n(s_1)\widetilde{V}_2\eta,\widetilde{V_1}\eta\ket)=\|s_1\|_n$ implies $\bra \pi_n(s_1)\widetilde{V}_2\eta,\widetilde{V_1}\eta\ket=\|s_1\|_n$, since $\pi_n$ is contractive.

Using the condition of equality in Cauchy-Schwarz inequality and contractivity of $\pi_n$, the equation $\bra \pi_n(s_1)\widetilde{V}_2\eta,\widetilde{V_1}\eta\ket=\|s_1\|_n$ would imply 
\begin{equation}\label{eq2}\widetilde{V}_1\eta=\frac{1}{\|s_1\|_n}\pi_n(s_1)\widetilde{V}_2\eta.
\end{equation}

Now, we consider $\phi : \mathcal S\rightarrow M_n$ as 
\[
\phi(x)=V_2^*\pi(x)V_2\quad\text{ for all }\quad x\in\mathcal S.
\]
Then, $\phi$ is completely positive map. Since $V_2$ is an isometry and $\pi$ is $\ast$-homomorphism, 
we get $\phi\in S_n(\mathcal S)$. Furthermore, 
\[
\frac{1}{\|s_1\|_n}\bra \phi_n(s_2)\eta, \phi_n(s_1)\eta\ket=\frac{1}{\|s_1\|_n}\bra \pi_n(s_2)\widetilde{V}_2\eta, \pi_n(s_1)\widetilde{V}_2\eta\ket=\bra \pi_n(s_2)\widetilde{V}_2\eta, \widetilde{V}_1\eta\ket,
\]
where we have used \eqref{eq2} in the last equation. The above equation along with \eqref{eq1} yield
\[
\lim\limits_{t\rightarrow 0^+}\dfrac{\|s_1+ts_2\|_n-\|s_1\|_n}{t}=\frac{1}{\|s_1\|_n}\bra \phi_n(s_2)\eta, \phi_n(s_1)\eta\ket.
\]
Furthermore, since 
$$\bra \phi_n(s_1)\eta, \phi_n(s_1)\eta\ket =\bra \pi_n(s_1)\widetilde{V}_2\eta, \pi_n(s_1)\widetilde{V}_2\eta\ket=\|s_1\|_n\bra \pi_n(s_1)\widetilde{V}_2\eta, \widetilde{V}_1\eta\ket=\|s_1\|_n^2,$$ 
we get $\|\phi_n(s_1)\eta\|=\|s_1\|_n$. Since $\|\phi_n\|\leq 1$ and $\|\eta\|=1$, we get $\|\phi_n(s_1)\|=\|s_1\|_n$. 
This proves 
\begin{align*}
	\lim\limits_{t\rightarrow 0^+}\dfrac{\|s_1+ts_2\|_n-\|s_1\|_n}{t} \leq  \frac{1}{\|s_1\|_n}\max\{& \mathrm{Re}(\bra\phi_n(s_2)\eta, \phi_n(s_1)\eta\ket) \,\, |\,\,  \phi \in S_n(\mathcal S),\eta\in\mathbb C^{n^2}\\
	&\|\eta\|=1,\|\phi_n(s_1)\|=\|s_1\|_n=\|\phi_n(s_1)\eta\|\}.
\end{align*}

To prove \eqref{e:qaz}, we need to prove the other side inequality, the proof of which is given below for completeness. Let $\phi\in S_n(\mathcal S)$ and $\eta\in \mathbb C^{n^2}$ with $\|\eta\|=1$ and  $$\|\phi_n(s_1)\|=\|s_1\|_n=\|\phi_n(s_1)\eta\|.$$ Since we have $\|\phi_n\|=1$, this implies $$\dfrac{\|s_1+ts_2\|_n-\|s_1\|_n}{t} \geq \dfrac{\|\phi_n(s_1+ts_2)\|-\|\phi_n(s_1)\|}{t} =\dfrac{\|A_1+tA_2\|-\|A_1\|}{t},$$ where $A_1:=\phi_n(s_1)\in M_n(M_k(\mathcal S))$ and  $A_2:=\phi_n(s_2)\in M_n(M_k(\mathcal S))$. This gives
\begin{align*}
	&\hspace{1.5cm} \dfrac{\|s_1+ts_2\|_n-\|s_1\|_n}{t} 
	\geq  \dfrac{\|A_1^*A_1+t(A_1^*A_2+A_1A_2^*)+t^2A_2^*A_2)\|-\|A_1\|^2}{t(\|A_1+tA_2\|+\|A_1\|)}\\
	&\geq \dfrac{1}{\|A_1+tA_2\|+\|A_1\|}\left[\dfrac{\mathrm{Re}(\bra A_1^*A_1\eta, \eta\ket)-\|A_1\|^2}{t}+2\ \mathrm{Re}(\bra A_1\eta, A_2\eta\ket)+t\bra A_2\eta, A_2\eta\ket \right]\\
	&=\dfrac{1}{\|A_1+tA_2\|+\|A_1\|}\left[2\ \mathrm{Re}(\bra A_1\eta, A_2\eta\ket)+t\bra A_2\eta, A_2\eta\ket \right],
\end{align*}
where in the last step we used the fact that $\|A_1\eta\|=\|A_1\|$. Finally, taking the limit $t\rightarrow 0$ on the both sides, we get $$\lim\limits_{t\rightarrow 0^+}\dfrac{\|s_1+ts_2\|_n-\|s_1\|_n}{t}\geq \mathrm{Re}(\bra A_1\eta, A_2\eta\ket).$$

Since $\phi\in S_n(\mathcal S)$ and $\eta\in \mathbb C^{n^2}$ with the conditions $\|\eta\|=1$ and  $\|\phi_n(s_1)\|=\|s_1\|_n=\|\phi_n(s_1)\eta\|$ were arbitrary, this 
implies the quantity
on the right hand side of \eqref{e:qaz} is no larger than the quantity on the left hand side of \eqref{e:qaz}.

 This completes the proof.
 \qed

\section{Applications of the Main Results}
\label{section3}

As an application of Theorem \ref{cor3}, we get the following special case of \cite[Theorem 2.4]{keckic2005}. Let $\|\cdot\|$ denotes the operator norm on $M_n$.

\begin{corollary}\label{cor1}
	Let $A, B\in M_n$. Then, $$\lim_{t\rightarrow 0^+}\dfrac{\|A+tB\|-\|A\|}{t}=\frac{1}{\|A\|}\max\{\mathrm{Re}(\bra B\eta, A\eta\ket) \, |\, \eta\in\mathbb C^n,\   \|\eta\|=1  \text{ and } \|A\eta\|=\|A\|\}.$$
\end{corollary}
\begin{proof} 
	We use Theorem \ref{cor3} with $V=\mathbb C$. Hence, there exists $\phi\in S_n(\mathbb C)$ and $\gamma\in (\mathbb C^n)^n$ such that $$\|\phi_n(A)\|=\|A\|=\|\phi_n(A)\gamma\|$$
	and \begin{equation}\label{eq:89}\lim_{t\rightarrow 0^+}\dfrac{\|A+tB\|-\|A\|}{t}=\frac{1}{\|A\|}\mathrm{Re}(\bra \phi_n(B)\gamma, \phi_n(A)\gamma\ket).\end{equation}
	We note that for $\phi\in S_n(\mathbb C)$, we have $\phi(\lambda)=\lambda I_n$ for all $\lambda\in\mathbb C$ and $\phi_n(C)=C\tens{}I_n$ for all $C\in M_n$. Consequently, we have $$\lim_{t\rightarrow 0^+}\dfrac{\|A+tB\|-\|A\|}{t}=\frac{1}{\|A\|}\mathrm{Re}(\bra (B\tens{}I_n)\gamma, (A\tens{}I_n)\gamma\ket).$$ Then we can decompose $\gamma$ as $(\eta_1,\dots,\eta_n)\in(\mathbb C^n)^n$ such that $\sum_{i=1}^n\|\eta_i\|^2=1$, $$\bra (B\tens{}I_n)\gamma, (A\tens{}I_n)\gamma\ket=\sum_{i=1}^n \bra B\eta_i, A\eta_i\ket\quad\text{and}\quad \|A\|=\|A\tens{}I_n)\gamma\|=\left(\sum_{i=1}^n\|A\eta_i\|^2\right)^{1/2}.$$ Now we have $$\|A\|^2=\sum_{i=1}^n\|A\eta_i\|^2\leq  \|A\|^2\sum_{i=1}^n\|\eta_i\|^2=\|A\|^2,$$ which  implies $\|A\eta_i\|=\|A\|\|\eta_i\|$. Thus, $$\bra (B\tens{}I_n)\gamma, (A\tens{}I_n)\gamma\ket=\sum_{\eta_i\neq 0} \bra B\eta_i, A\eta_i\ket=\sum_{\eta_i\neq 0}\|\eta_i\|^2\left\bra B\frac{\eta_i}{\|\eta_i\|}, A\frac{\eta_i}{\|\eta_i\|}\right\ket,$$ lies in the convex hull of numerical range of $(A^\ast B)\tens{} I_n$ restricted to the subspace $\{\eta\in \mathbb C^n : \|A\eta\|=\|A\| \}$. Using Haudorff-Toeplitz theorem, there exists $\eta\in\mathbb C^n$ such that $\bra (B\tens{}I_n)\gamma, (A\tens{}I_n)\gamma\ket=\bra B\eta, A\eta\ket$. Thus, \eqref{eq:89} implies $$\lim_{t\rightarrow 0^+}\dfrac{\|A+tB\|-\|A\|}{t}=\frac{1}{\|A\|}\mathrm{Re}(\bra (B\tens{}I_n)\gamma, (A\tens{}I_n)\gamma\ket)=\frac{1}{\|A\|}\mathrm{Re}(\bra B\eta, A\eta\ket).$$ This completes the proof.
\end{proof}

Gateaux derivatives are closely associated with a notion of orthogonality in a normed space, known as Birkhoff-James orthogonality \cite{james1947}. Let $\mathcal X$ be a vector space. Let $x\in \mathcal X$ and $\mathcal Y$ be a subspace of $X$. Then, $x$ is said to be (Birkhoff-James) orthogonal to $\mathcal Y$ if and only if $$\|x\|\leq \|x-y\|\quad\text{ for all }\quad y\in \mathcal Y.$$ When $\mathcal Y=\mathbb Cy$ is one dimensional subspace generated by $y\in \mathcal X$, $x$ is orthogonal to $\mathcal Y$ is abbreviated as $x$ is orthogonal to $y$. It was proved in \cite[Proposition 1.5]{keckic2005} that $x$ is orthogonal to $y$ if and only if $\inf_{\varphi\in[0,2\pi]}D_{\varphi, x}(y)\geq 0$, where $$D_{\varphi, x}(y)=\lim_{t\rightarrow 0^+}\dfrac{\|x+te^{i\varphi}y\|-\|x\|}{t}.$$ Using this fact, Corollary \ref{cor1}, and along the lines of \cite[Corollary 3.1]{keckic2005}, we get the following.

\begin{corollary}\label{BS}
	Let $A\in M_n$, then $A$ is orthogonal to $B$ if and only if there exists $\eta\in \mathbb C^n$ such that $\|\eta\|=1$, $\bra A\eta, \eta\ket=\|A\|$ and $\bra B\eta, \eta\ket =0$.
\end{corollary}

\begin{remark} This theorem has become famous by the name of Bhatia-\v Semrl Theorem after this appeared in \cite[Theorem 1.1]{bhatia--semrl1999} in the following form: For $B\in M_n$, we have $\|A\|\leq \|A+\lambda B\|$ for all $\lambda\in\mathbb C$ if and only if there exists $\eta\in \mathbb C^n$ such that $\|\eta\|=1$, $\bra A\eta, \eta\ket=\|A\|$ and $\bra B\eta, \eta\ket =0$.
	
This theorem is actually a special case of \cite[Lemma 2.2]{magajna1993}, that was motivated by the proof of \cite[Theorem 2]{stampfli1970} and \cite[Lemma 9.14]{Davidson-book-nestalgebras}.
\end{remark}

Now, we see a generalization of this form of Bhatia-\v Semrl Theorem in the setting of operator spaces.
\begin{theorem}\label{thm3}
	Let $(V, \|\cdot\|_n)$ be an operator space. Let $v\in M_n(V)$ and $\mathcal W$ be a subspace of $M_n(V)$. Then, we have $$\|v\|_n\leq \|v-w\|_n\quad\text{ for all }\quad w\in\mathcal W$$ if and only if there is a completely contractive mapping $\phi : V\rightarrow M_n$ and a unit vector $\eta\in(\mathbb C^n)^n$ such that \begin{equation}\label{eq:90}\|\phi_n(v)\eta\|=\|v\|_n\quad\text{and}\quad \bra\phi_n(v)\eta, \phi_n(w)\eta\ket=0\quad\text{for all }\quad w\in\mathcal W.\end{equation}
\end{theorem}
\begin{proof} If $\phi : V\rightarrow M_n$ is a completely contractive mapping and  $\eta\in(\mathbb C^n)^n$ is a unit vector such that \eqref{eq:90} holds. Then, \begin{align*}
		\|v\|_n^2&=\bra\phi_n(v)\eta, \phi_n(v)\eta\ket\\
		&\leq \bra\phi_n(v)\eta, \phi_n(v)\eta\ket+\bra\phi_n(w)\eta, \phi_n(w)\eta\ket\\
		&=\bra\phi_n(v+w)\eta, \phi_n(v+w)\eta\ket\\
		&\leq \|v+w\|_n^2.
	\end{align*}
	Now, we prove the other direction. Let $\|v\|_n\leq \|v-w\|_n$ for all $w\in\mathcal W$. 
		Using \cite[Theorem 2.1]{james1947},  there exists a functional $f\in M_n(V)^*$ such that $\|f\|=1$, $f(v)=\|v\|$ and $f(w)=0$ for all $w\in \mathcal W$. Using \cite[Lemma 2.3.3]{Effros--Ruan-book}, there exists a complete contraction $\phi : V\rightarrow M_n$ and unit vectors $\xi,\eta\in(\mathbb C^n)^n$ such that \begin{equation}\label{equation1} f(u)=\bra\phi_n(u)\eta, \xi\ket,\end{equation} for all $u\in M_n(V)$. Now, a simple application of the condition of equality in Cauchy-Schwarz inequality implies that for a functional of the form \eqref{equation1}, the condition  $f(v)=\|v\|_n$ implies $\xi=\frac{1}{\|v\|_n}\phi_n(v)\eta.$ Combining all these statements give the required result.
\end{proof}

The above corollary can be used to the existence of some special states in operator systems and $C^*$-algebras. We recall \cite{Rockafellar-book} that a norm's face $F$ of $\mathcal S$  is a convex subset of the unit ball of $\mathcal S$ such that if for $s_1,s_2\in\mathcal S$ with $\|s_1\|=\|s_2\|=1$ and $\lambda s_1+(1-\lambda)s_2\in F$ with $0<\lambda<1$, then $s_1,s_2\in F$. A norm's face is maximal if it is not properly contained in any other norm's face. Let $\mathrm{ri}(\mathcal F)$ denotes the relative interior of a norm's face.

\begin{corollary}\label{cor6}
		Let $\mathcal S$ be an operator system with Archimedean order unit $e$. Let $\mathcal F$ be the maximal norm's face of $\mathcal S$ containing $e$. If $\mathcal W$ is a subspace of $\mathcal S$ such that 		
		 $$1\leq \|e- w\|\quad\text{ for all }\quad w\in \mathcal W,$$  then there exists a matrix state $\phi\in S_1(\mathcal S)$ such that $\phi(v)=\|v\|$ for all $v\in\mathrm{ri}(\mathcal F)$ and $0\in W_1(\phi(w))$ for all $w\in\mathcal W$.
\end{corollary} 
\begin{proof}
	Using Theorem \ref{thm3}, we can get a state $\phi\in S_1(\mathcal S)$ such that $0\in W_1(\phi(w))$ for all $w\in\mathcal W$. We have, $\phi(e)=1$. If we prove $\phi(v)=\|v\|$ for all $v\in\mathrm{ri}(\mathcal F)$,  we will be done. Let $v\in\mathrm{ri}(\mathcal F)$. Using \cite[Theorem 6.2]{Rockafellar-book},  there exists $\epsilon>0$ such that $$\|(1-t)e+tv\|=1\text{ for all }t\in [0,1+\epsilon].$$ Then, $$|t\phi(v)+(1-t)|=|\phi((1-t)e+tv)|\leq \|(1-t)e+t v\|=1,$$ for all $t\in [1-\epsilon,1+\epsilon]$. Taking $t=1\pm \epsilon$, we get $|1\pm \epsilon(f(v)-1)|\leq 1$, which gives $f(v)=1$.
\end{proof}

For a $C^*$-algebra $\mathcal A$, a slightly different version of the above corollary also holds.

\begin{theorem}\label{thm5}
		Let $\mathcal A$ be a unital $C^*$-algebra with unit element $1_{\mathcal A}$. Let $x\in\mathcal A$ be an element of norm one and $\mathcal F$ be the maximal norm's face of $\mathcal A$ containing $x$. If $\mathcal B$ is a subspace of $\mathcal A$ such that $$\|x\|\leq \|x-y\|\quad\text{ for all }\quad y\in \mathcal B,$$  then there exists a state $\phi\in S_1(\mathcal A)$ such that $$\phi(x^*x)=1\text{ and } \phi(x^*y)=0\quad \text{ for all }\quad x\in\mathrm{ri}(\mathcal F)\text{ and }y\in \mathcal B.$$
\end{theorem}
\begin{proof}
	By Hahn-Banach theorem, there exists a functional $f\in\mathcal A^*$ such that $\|f\|=1=f(x)$ and $f(y)=0$ for all $y\in\mathcal B$. Using the argument in Corollary \ref{cor6}, we have $f(x)=1$ for all $x\in\mathrm{ri}(\mathcal F)$. For functional $f$, there exists a cyclic representation $(\mathcal H, \pi, \xi)$ with $\|\xi\|=1$ and a unit vector $\eta\in\mathcal H$ such that $$f(z)=\bra\pi(z)\xi, \eta\ket\quad\text{ for all }z\in\mathcal A.$$ Since $f(x)=1$ , we get $\bra\pi(x)\xi, \eta\ket=\|x\|$ for all $x\in\mathrm{ri}(\mathcal F)$. Using the condition of Cauchy-Schwarz inequality,
	we get $\eta=\pi(x)\xi$ for all $x\in\mathrm{ri}(\mathcal F)$. If we take $$\phi(z)=\bra \pi(z)\xi,\xi\ket\quad\text{ for all}\quad z\in\mathcal A,$$ then $\phi$ is the required state on $\mathcal A$.
\end{proof}

Finally, we see matricial versions of some known results. We recall the notion of positive operator-valued measure (POVM) (see \cite{farenick--plosker--smith2011} for more details). Assume that $X$ is a compact Hausdorff space, and $\mathcal O(X)$ is the $\sigma$-algebra of Borel sets of $X$. Let $\mathcal B(\mathcal H)$ denotes the space of bounded linear operators on a Hilbert space $\mathcal H$. A function $\nu : \mathcal O(X)\rightarrow \mathcal B(\mathcal H)$ is a POVM on $X$ if \begin{enumerate}
	\item[(i)] $\nu(E)\in\mathcal B(\mathcal H)_+$ with $0\leq \lambda\leq 1$ for every eigenvalue $\lambda$ of $\nu(E)$;
	\item[(ii)] for every countable collection $\{E_k\}_{k\in\mathbb N}\subseteq\mathcal O(X)$ with $E_k\cap E_{k'}=\emptyset$ for $k\neq k'$, $$\nu\bigg(\cup_{n\in\mathbb N}E_k\bigg)=\sum_{k\in\mathbb N}\nu(E_k),$$ where the convergence 
	of the right-hand side of the equality above is with respect to the weak
	operator topology of $\mathcal B(\mathcal H)$.
\end{enumerate}
If a POVM $\nu$ on $\mathcal O(X)$ satisfies $\nu(X)=I_{\mathcal B(\mathcal H)}$, then it is referred to as a \emph{quantum probability measure}. If $\nu:\mathcal O(X)\rightarrow M_k$ is a POVM, then, for each pair $\xi,\nu\in\mathbb C^k$, the formula
\[
\mu_{\xi,\eta}^{[\nu]}(E)=\langle \nu(E)\xi,\eta\rangle, \mbox{ for }E\in\mathcal O(X),
\]
defines a complex Borel measure on $\mathcal O(X)$. Hence, for each $f\in C(X)$, we denote by $\displaystyle\int_Xf\,d\nu$
the unique operator on $\mathbb C^k$ for which
\[
\left\langle \left(\displaystyle\int_Xf\,d\nu\right)\xi,\eta\right\rangle= \int_X f\,d\mu_{\xi,\eta}^{[\nu]}, \mbox{ for all }\xi,\eta\in\mathbb C^k.
\] 
In \cite{farenick--plosker--smith2011}, the notion of integral of $f\in M_k(C(X))$ with respect to $\nu$ was also defined. Through the identification of the C$^\ast$-algebra $M_k(C(X))$ with the unital homogeneous C$^\ast$-algebra of all continuous functions $f : X \rightarrow M_k$, the integral $\int_X f\ d\nu$ was defined as the unique operator on $\mathbb C^k$ having the property that $$\mathrm{trace}\left(\rho\int_X f\ d\nu\right)=\int_X f_\rho\ d\nu,$$ where for all $x\in X$, we have $f_\rho(x)=\mathrm{trace}\left(\rho\bigg(\dfrac{d\nu}{d\mu}(x)\bigg)^{1/2}f(x)\bigg(\dfrac{d\nu}{d\mu}(x)\bigg)^{1/2}\right)$, and $\dfrac{d\nu}{d\mu}$ represents the principle Radon-Nikod\'{y}m of $\nu$ with respect to its induced classical probability measure as defined in  \cite[\S2]{farenick--plosker--smith2011}. We refer the readers to \cite{farenick--plosker--smith2011} for more details. Such quantum probability measures occur at various places, for example as completely positive Fr\'{e}chet derivatives \cite{miklos2016} which have applications in constructing operator means \cite{miklos2013, miklos2014}.

As an application of Theorem \ref{thm3}, we get the following extension of \cite[Theorem 1.3]{Singer-book} in the setting of POVMs. (see also \cite[Corollary 2.1]{keckic2012}).

\begin{corollary}\label{doug}
	If $X$ is a compact Hausdorff space and $Id_n$ denotes the diagonal matrix in $M_n(C(X))$ with diagonal entries $1_X$. The, for a subspace $\mathcal W$  of  $M_n(C(X))$, we have
\[ \|Id_n\|_{M_n(C(X))}\leq \|Id_n- G\|_{M_n(C(X))}, \mbox{ for all }G\in \mathcal W,
\]
if and only if there exists a quantum probability measure 
$\nu : \mathcal O(X)\rightarrow M_n\otimes M_n$ and a unit vector $\eta\in\mathbb C^n\otimes\mathbb C^n$ such that
$\left\langle \left(\displaystyle\int_X G\,d\nu\right)\eta,\eta\right\rangle=0$, for all $G\in\mathcal W$.
\end{corollary}
\begin{proof}
	Assume $\nu : \mathcal O(X)\rightarrow M_n\otimes M_n$ is a quantum probability and $\eta\in\mathbb C^n\otimes \mathbb C^n$ 
	is a unit vector
	for which $\langle (\int_X G\,d\nu)\eta,\eta\rangle=0$, for all $G\in\mathcal W$. 
	If $\{\epsilon_i\}_{i=1}^n$ denotes the canonical orthonormal basis for $\mathbb C^n$, then there exist $\eta_1,\dots,\eta_n\in\mathbb C^n$
	such that
	\begin{equation}\label{eq:94}
	\eta = \sum_{i=1}^n \eta_i\otimes \epsilon_i.
	\end{equation}
	Let $W:\mathbb C^n\rightarrow \mathbb C^n \otimes \mathbb C^n$ be the linear map for which $W\epsilon_i=\eta_i\otimes \epsilon_i$, for each $i$.
	Let $\xi=\epsilon_1+\dots+\epsilon_n$, then  
	\begin{equation}\label{eq:95}W\xi=\eta.\end{equation} The map $W^*\nu W:\mathcal O(X)\rightarrow M_n$ is a POVM and has the property that, for every $E\in\mathcal O(X)$
	and $\gamma,\delta\in\mathbb C^n$,
	\[
	\mu_{\gamma,\delta}^{[W^*\nu W]}(E)=\langle W^*\nu(E) W\gamma,\delta \rangle = \mu_{W\gamma,W\delta}^{[\nu]}(E).
	\]
	In particular, using \eqref{eq:94} and \eqref{eq:95}, we have
	\[
	\mu_{\xi,\xi}^{[W^*\nu W]}(E) = \mu_{\eta,\eta}^{[\nu]}(E).
	\]
	Now define a linear map $\phi:C(X)\rightarrow M_n$ by
	\[
	\phi(f)=\int_X f\,d(W^*\nu W).
	\]
	Note that $\phi$ is a matrix state (see \cite[Theorem 2.5]{farenick--plosker--smith2011}). If $F=[f_{ij}]_{i,j}\in C(X)\otimes M_n$, then
	\[
	\begin{array}{rcl}
	\langle \phi_n(F)\eta,\eta\rangle &=& \displaystyle\sum_{i,j=1}^n \langle \phi(f_{ij})\eta_j,\eta_i\rangle \\ && \\
	&=& \displaystyle\sum_{i,j=1}^n \displaystyle\int_X f_{ij} \, d\mu_{W\epsilon_j,W\epsilon_i}^{[\nu]} \\ && \\
	&=& \left\langle\left( \displaystyle\int_X F \, d\nu\right)\eta,\eta \right\rangle.	
	\end{array}
	\]
	Hence, $\langle\phi_n(G)\eta,\eta\rangle=0$, for every $G\in\mathcal W$, and so, by Theorem \ref{thm3},
	$$\|Id_n\|_{M_n(C(X))}\leq \|Id_n- G\|_{M_n(C(X))}, \mbox{for every }G\in\mathcal W.$$
	
	Conversely, assume that $\|Id_n\|_{M_n(C(X))}\leq \|Id_n- G\|_{M_n(C(X))}$, for every $G\in\mathcal W$. As an application of Theorem \ref{thm3}, there exist a 
	matrix state $\phi:C(X)\rightarrow M_n$ and a unit vector $\eta\in\mathbb C^n\otimes\mathbb C^n$ such that
	$\langle\phi_n(G)\eta,\eta\rangle=0$, for all $G\in\mathcal W$. As noted in \cite[\S4]{Paulsen-book}, the matrix state $\phi$ is 
	determined by a unique quantum probability measure
	$\omega:\mathcal O(X)\rightarrow M_n$ whereby, for each pair $\gamma,\delta\in \mathbb C^n$,
	\[
	\langle\phi(f)\gamma,\delta\rangle =\int_X f\,d\mu_{\gamma,\delta}^{[\omega]},
	\mbox{ for all } f\in C(X).
	\]
	(The POVM $\omega$ has the form $\omega(E)=V^* P(E) V$, where $V^*\pi V$ is a minimal
	Stinespring dilation of $\phi$ and $P$ is the spectral measure on $\mathcal O(X)$ arising from the Spectral Theorem applied to the abelian C$^*$-subalgebra
	$\pi\left(C(X)\right)$ of $\mathcal B(\mathcal H_\pi)$.) Define an operator-valued measure $\nu:\mathcal O(X)\rightarrow M_n\otimes M_n$ by
	\[
	\nu(E)=\sum_{i=1}^n\sum_{j=1}^n \omega(E)\otimes E_{ij},
	\]
	where $\{E_{ij}\}_{i,j=1}^n$ are the canonical matrix units of $M_n$. Note that $\nu$ is positive, as $\omega$ is positive; hence 
	$\nu$ is a quantum probability measure.
	
	Express $\phi_n$ as $\phi_n=\phi\otimes\mbox{\rm Id}_{M_n}$, so that, for every $F=(f_{ij})\in C(X)\otimes M_n$, we have
	\[
	\phi_n(F)=\sum_{i,j=1}^n \left(\int_X f_{ij}\,d\,\omega\right)\otimes E_{ij} = 	\int_X F\,d\,\nu.
	\]
	Because $\langle\phi_n(G)\eta,\eta\rangle=0$ if $G\in \mathcal W$, we deduce 
	$\left\langle \left(\displaystyle\int_X G\,d\nu\right)\eta,\eta\right\rangle=0$, for all $G\in\mathcal W$.
\end{proof}

Along the lines of proof of Corollary \ref{doug} and Theorem \ref{cor3}, we also get the following extension of \cite[Theorem 2.1]{keckic2012} for the Gateaux derivative in case of commutative $C^*$-algebras.

\begin{corollary}
	Let $X$ be a compact Hausdorff space, and $C(X)$ be the space of all continuous function on $X$. Let $F=(f_{ij}), G=(g_{ij})\in M_n(C(X))$, then \begin{align*}\lim\limits_{t\rightarrow 0^+}\dfrac{\|F+tG\|_n-\|F\|_n}{t}=\max\bigg\{&\mathrm{Re}\left\langle \left(\displaystyle\int_X G\,d\nu\right)\eta,\left(\displaystyle\int_X F\,d\nu\right)\eta\right\rangle\,\, | \,\, \nu \text{ is a quantum}\\
		&\text{probability measure with } \left\|\displaystyle\int_X F\,d\nu\right\|=\|F\|_{M_n(C(X))}\bigg\}.\end{align*}
\end{corollary}

We end our article with a remark that there are notions of real operator spaces and real operator systems, and all our results also extend naturally when the underlying field is $\mathbb R$. 

\subsection*{Acknowledgements}

The author is extremely thankful to the referee for a careful reading of the article, and providing many useful suggestions that has significantly improved the presentation of the article.  The author would like to thank Doug Farenick for many useful discussions. The author is also thankful for the great hospitality service provided by BIRS (Banff, Canada) during the workshop `Operator Systems and their Applications (25w5405)', where the idea of the project originated.

 

\end{document}